\documentclass{amsart}%
\usepackage[latin9]{inputenc}
\usepackage{amstext}
\usepackage{amsthm}
\usepackage{amssymb}
\usepackage{amsfonts}
\usepackage[pdftex]{graphicx}
\usepackage{amsmath}%
\providecommand{\U}[1]{\protect\rule{.1in}{.1in}}
\providecommand{\U}[1]{\protect\rule{.1in}{.1in}}
\newtheorem{theorem}{Theorem}
\theoremstyle{plain}

\newtheorem{corollary}{Corollary}

\newtheorem{definition}{Definition}
\newtheorem{example}{Example}

\newtheorem{lemma}{Lemma}

\newtheorem{proposition}{Proposition}
\newtheorem{remark}{Remark}

\numberwithin{equation}{section}
\begin{document}
\title{Old and New on Strongly Subadditive/Superadditive Functions}
\author{Constantin P. Niculescu}
\address{Department of Mathematics, University of Craiova, Craiova 200585, Romania}
\email{constantin.p.niculescu@gmail.com}
\dedicatory{Dedicated to Professor Lars-Erik Persson, on the occasion of his 80th anniversary}
\date{January 8, 2025}

\begin{abstract}
In this paper we provide insight into the classes of strongly
subadditive/superadditive functions by highlighting numerous new examples and
new results.

\end{abstract}
\maketitle

\section{Introduction}

While the classes of subadditive and superadditive functions are ubiquitous in
many fields of mathematics, their strong companions seem to be rarely
mentioned in the existing literature. A notable exception is the old theorem
of Lieb and Ruskai \cite{Lieb1973} stating the strong subadditivity of von
Neumann quantum entropy.

The aim of the present paper is to reveal the existence of a wide range of
strongly subadditive/superadditive functions, motivated by potential theory,
probability theory and statistics, combinatorial optimization, risk
management, quantum physics, etc. As a basis we used the updated 2025 version
of the book \cite{NP2025}, published in collaboration with Lars-Erik Persson.

For convenience, we will restrict ourselves to the case of a real-valued
function $\Phi$ defined on a \emph{convex cone} $\mathcal{C}$ of a real linear
space $E.$ This means that $a\mathbf{x}+b\mathbf{y}\in\mathcal{C}$ whenever
$a,b>0$ and $\mathbf{x},\mathbf{y}\in\mathcal{C};$ the origin does not have to
belong to $\mathcal{C}$ but this occurs when $E$ is a Banach space and
$\mathcal{C}$ is supposed to be closed.

Two important cones in the Euclidean space $\mathbb{R}^{N}$ endowed with the
coordinatewise ordering are
\[
\mathbb{R}_{+}^{N}=\left\{  \mathbf{x}=(x_{1},..,x_{N}):x_{k}\geq0\text{ for
}1\leq k\leq N\right\}
\]
and
\[
\mathbb{R}_{++}^{N}=\left\{  \mathbf{x}=(x_{1},..,x_{N}):x_{k}>0\text{ for
}1\leq k\leq N\right\}  .
\]
Notice that $\mathbb{R}^{N}$ itself represents an example of convex cone
(though not associated to an order relation). For more examples of ordered
Banach spaces and convex cones see the Appendix at the end of this paper.

A function\ $\Phi:\mathcal{C}\rightarrow\mathbb{R}$ is \emph{subadditive }if%
\begin{equation}
\Phi(\mathbf{x+y})\leq\Phi(\mathbf{x})+\Phi(\mathbf{y}) \label{eq_sub}%
\end{equation}
for all $\mathbf{x,y\in\,}\mathcal{C}$ and it is called \emph{strongly}
\emph{subadditive} if in addition%
\begin{equation}
\Phi(\mathbf{x+y+z})+\Phi(\mathbf{z})\leq\Phi(\mathbf{x+z})+\Phi(\mathbf{y+z})
\label{eq_str_sub}%
\end{equation}
for all $\mathbf{x,y,z}\in\mathcal{C}.$

In general, the conditions (\ref{eq_sub}) and (\ref{eq_str_sub}) are
altogether independent. There are subadditive functions that are not strongly
subadditive (see the case of the function $1/x$ restricted to $\mathbb{R}%
_{++})$ as well as functions that verify the inequality (\ref{eq_str_sub})
that are not subadditive (since (\ref{eq_str_sub}) is invariant under the
addition of any affine function to $\Phi$, while the property of subadditivity
fails this property).

If $0\in\mathcal{C},$ the inequality (\ref{eq_sub}) implies $\Phi(0)\geq0$ and
in this case the strong subadditivity of the function $\Phi$ is equivalent to
the fact that $\Phi(0)\geq0$ and $\Phi$ verifies the condition
(\ref{eq_str_sub}).

The function $\Phi:\mathcal{C}\rightarrow\mathbb{R}$ is called
\emph{superadditive }(respectively \emph{strongly superadditive})\emph{ }if
$-\Phi$ is subadditive (respectively strongly subadditive). Thus, if
$0\in\mathcal{C},$ the function $\Phi$ is strongly superadditive if
$\Phi(0)\leq0$ and
\begin{equation}
\Phi(\mathbf{x+y+z})+\Phi(\mathbf{z})\geq\Phi(\mathbf{x+z})+\Phi(\mathbf{y+z})
\label{eq_str_super}%
\end{equation}
for all $\mathbf{x,y,z}\in\mathcal{C}$.

The consideration of difference operators offers more insight about the nature
of strongly subadditive/superadditive functions. Recall that the difference
operators $\Delta_{\mathbf{x}}$ $(\mathbf{x}\in\mathcal{C})$ act on the
functions defined on the cone $\mathcal{C}$ via the formula%
\[
\Delta_{\mathbf{x}}\Phi(\mathbf{z})=\Phi(\mathbf{x}+\mathbf{z})-\Phi
(\mathbf{z})\text{\quad for all }\mathbf{z}\in\mathcal{C}.
\]

The operators $\Delta_{\mathbf{x}}$ are linear and commute to each other.
Important for us are the \emph{differences of second order,} which are given
by the formula%
\[
\Delta_{\mathbf{x}}\Delta_{\mathbf{y}}\Phi(\mathbf{z})=\Phi(\mathbf{x}%
+\mathbf{y}+\mathbf{z})-\Phi(\mathbf{x}+\mathbf{z})-\Phi(\mathbf{y}%
+\mathbf{z})+\Phi(\mathbf{z})\text{\quad for all }\mathbf{x},\mathbf{y}%
,\mathbf{z}\in\mathcal{C}.
\]

Thus a function $\Phi:\mathcal{C}\rightarrow\mathbb{R}$ is strongly
subadditive if and only if it is subadditive and%
\begin{equation}
\Delta_{\mathbf{x}}\Delta_{\mathbf{y}}\Phi(\mathbf{z})\leq0\text{\quad for all
}\mathbf{x},\mathbf{y},\mathbf{z}\in\mathcal{C}. \label{eq_concave}%
\end{equation}
Similarly, $\Phi$ is strongly superadditive if and only if it is superadditive
and%
\begin{equation}
\Delta_{\mathbf{x}}\Delta_{\mathbf{y}}\Phi(\mathbf{z})\geq0\text{\quad for all
}\mathbf{x},\mathbf{y},\mathbf{z}\in\mathcal{C}. \label{eq_convex}%
\end{equation}

Notice that the inequalities (\ref{eq_concave}) and (\ref{eq_convex}) relate
the subject of strong sub/super additivity to that of higher-order
monotonicity (more precisely, to the 2-monotone decreasing and 2-monotone
increasing functions). See Ressel \cite{Res2014} and Niculescu \cite{N2019}.

In the framework of one real variable continuous functions, the inequalities
(\ref{eq_concave}) and (\ref{eq_convex}) can be characterized in terms of convexity.

\begin{lemma}
\label{lem_HLP}A continuous function $f:\mathbb{R}_{+}\rightarrow\mathbb{R}$
is convex if and only if it verifies the inequality%
\[
\Delta_{\mathbf{x}}\Delta_{\mathbf{y}}f(z)=f(x+y+z)-f(x+z)-f(y+z)+f(z)\geq
0\text{\quad for all }x,y,z\in\mathbb{R}_{+}.
\]

\end{lemma}

\begin{proof}
If $f$ is convex, then the inequality $\Delta_{\mathbf{x}}\Delta_{\mathbf{y}%
}f(z)\geq0$ is an easy consequence of the Hardy-Littlewood-Pólya inequality of
majorization. See \cite{NP2025}, Theorem 6.1.3.

Conversely, assuming that $\Delta_{\mathbf{x}}\Delta_{\mathbf{y}}f(z)\geq0$
for all $x,y,z\in\mathbb{R}_{+},$ it follows that $f$ is midpoint convex and
thus the convexity of $f$ is a consequence of an well known result due to
Jensen that asserts that convexity is equivalent to the midpoint convexity
combined with the continuity at the interior points. See \cite{NP2025},
Theorem 1.1.8.
\end{proof}

We are thus led to the following characterization of strongly subadditive
functions defined on $\mathbb{R}_{+}:$

\begin{theorem}
\label{thm_str_sub_1var}Suppose that $\Phi:\mathbb{R}_{+}\rightarrow
\mathbb{R}$ is a continuous function. Then $\Phi$ is strongly subadditive if
and only if $\Phi$ is concave and $\Phi(0)\geq0$ $($respectively $\Phi$ is
strongly superadditive if and only if $\Phi$ is convex and $\Phi(0)\leq0).$
\end{theorem}

As a consequence, we infer that the following functions are strongly
subadditive on $\mathbb{R}_{+}$,
\begin{gather*}
mx+n+px^{\alpha}\text{ (for }\alpha\in\lbrack0,1]\text{, }m\in\mathbb{R}\text{
and }n,p\in\mathbb{R}_{+}),\text{ }1-\left(  1+\alpha x^{2}\right)
^{1/2}~(\text{for }\alpha>0),\\
-\left(  x+\alpha\right)  \log\left(  x+\alpha\right)  \text{ (for }\alpha
\in\lbrack0,1]\text{)},\\
\log(1+x),~~-\log(\cosh x),\quad e-(1+x)^{1/x},~1-e^{-x}~\text{and }%
(1+e^{-x})^{-1},
\end{gather*}
while their opposites are strongly superadditive. Also superadditive on
$\mathbb{R}_{+}$ are the functions%
\[
\frac{x^{2}}{2}\pm\log(1+x),\quad\frac{x^{2}}{2}\pm\sin x,\text{ }\frac{x^{2}%
}{2}\pm\cos x\text{\quad and }x\Gamma(x)-1.
\]

Theorem \ref{thm_str_sub_1var} can be refined by considering the case of
$\alpha$-strongly convex functions. Recall that a real-valued function $\Phi$
defined on a convex subset $C$ of $\mathbb{R}^{N}$ is called $\alpha
$-\emph{strongly convex} (that is, strongly convex with parameter $\alpha>0)$
if $\Phi-$ $(\alpha/2)\left\Vert \cdot\right\Vert ^{2}$ is convex, that is,%
\[
\Phi((1-\lambda)\mathbf{x}+\lambda\mathbf{y})\leq(1-\lambda)\Phi
(\mathbf{x})+\lambda\Phi(\mathbf{y})-\frac{1}{2}\lambda(1-\lambda
)\alpha\left\Vert \mathbf{x}-\mathbf{y}\right\Vert ^{2}%
\]
for all $\mathbf{x},\mathbf{y}$ in $C$ and $\lambda\in(0,1)$. The function
$\Phi$ is called $\alpha$-\emph{strongly concave} if $-\Phi$ is $\alpha
$-strongly convex, equivalently, if it verifies estimates of the form%
\[
\Phi((1-\lambda)\mathbf{x}+\lambda\mathbf{y})\geq(1-\lambda)\Phi
(\mathbf{x})+\lambda\Phi(\mathbf{y})+\frac{1}{2}\lambda(1-\lambda
)\alpha\left\Vert \mathbf{x}-\mathbf{y}\right\Vert ^{2}%
\]
for all $\mathbf{x},\mathbf{y}$ in $C$ and $\lambda\in(0,1)$.

A simple way to generate strongly convex functions on $\mathbb{R}_{+}$ is
provided by the formula%
\[
\Phi(x)=\int_{0}^{x}\left(  \int_{0}^{s}\varphi(t)dt\right)  ds,
\]
where $\varphi:$ $[0,\infty)\rightarrow\mathbb{R}$ is any strictly positive,
bounded and continuous function.

According to Theorem \ref{thm_str_sub_1var}, one obtains the following result:

\begin{proposition}
\label{prop_a-str_sub}If $f:\mathbb{R}_{+}\rightarrow\mathbb{R}$ is an
$\alpha$-strongly convex function such that $f(0)\leq0,$ then%
\[
f(x+y+z)+f(z)\geq f(x+z)+f(y+z)+\alpha xy
\]
for all $x,y,z\mathbf{\geq0}.$
\end{proposition}

The analogue of Proposition \ref{prop_a-str_sub} in the case of $\alpha
$-strongly concave functions can be deduced by replacing $f$ by $-f.$

\begin{remark}
\label{rem1}In the case of concrete functions, the property of strong
subadditivity may lead to unexpected results. For example, the property of
strong subadditivity of the functions $\frac{x^{2}}{2}\pm\log(1+x)$ leads to
the following double inequality combining polynomials and exponentials:%
\[
e^{xy}\geq\frac{\left(  1+z\right)  \left(  1+x+y+z\right)  }{\left(
1+x+z\right)  (1+y+z)}\geq e^{-xy}\text{\quad for all }x,y,z\geq0.
\]
\newline In the same vein, if $\Phi:\mathbb{R}_{+}\rightarrow\mathbb{R}$ is a
function of class $C^{1}$ such that%
\[
\left\vert \Phi^{\prime}(x)-\Phi^{\prime}(y)\right\vert \leq L\left\vert
x-y\right\vert
\]
for some positive constant $L>0,$ then%
\[
\left\vert \Phi(x+y+z)-\Phi(x+z)-\Phi(y+z)+\Phi(z)\right\vert \leq Lxy
\]
for all $x,y,z\geq0.$
\end{remark}

It is natural to wonder what is happening if $\mathbb{R}_{+}$ is replaced by
the positive cone $\mathbb{R}_{+}^{N}$ (of the ordered linear space
$\mathbb{R}^{N}$ endowed with the coordinatewise ordering).

Clearly, the sums of functions that depend on single variables such as the
\emph{Shannon entropy},%
\[
\Phi(\mathbf{x})=-\sum\nolimits_{k=1}^{n}x_{k}\log x_{k},\text{\quad
}\mathbf{x}\in\mathbb{R}_{+}^{n},
\]
have the same behavior as in the case $N=1$ (so the \emph{Shannon entropy} is
strongly subadditive).

Another example of the same nature is provided by the strongly superadditive
function
\[
\left\Vert \mathbf{x}\right\Vert ^{2}=x_{1}^{2}+\cdots+x_{N}^{2}%
\]
when restricted to $\mathbb{R}_{+}^{N}$. Alternatively, notice that this
function vanishes at the origin and
\[
\Delta_{\mathbf{x}}\Delta_{\mathbf{y}}\left(  \left\Vert \mathbf{z}\right\Vert
^{2}\right)  =2\langle\mathbf{x},\mathbf{y}\rangle\geq0\text{\quad for all
}\mathbf{x},\mathbf{y}\in\mathbb{R}_{+}^{N}.
\]

A simple computation shows that the scalar product $S(\mathbf{x}%
,\mathbf{y})=\langle\mathbf{x},\mathbf{y}\rangle$ is strongly superadditive on
the product cone $\mathbb{R}_{+}^{N}\times\mathbb{R}_{+}^{N}.$

Also straightforward is the case of functions of the form%
\[
\Phi(\mathbf{x})=f\left(  \langle\mathbf{x},\mathbf{a}\rangle\right)
,\text{\quad}\mathbf{x}\in\mathbb{R}_{+}^{N},
\]
associated to vectors $\mathbf{a}\in\mathbb{R}_{+}^{n}$ and continuous
functions $f:\mathbb{R}_{+}\rightarrow\mathbb{R}.$ According to the
majorization theorem of Hardy-Littlewood-Pólya, $\Phi$ is strongly subadditive
if $f$ is concave and $f(0)\geq0$ (respectively strongly superadditive if $f$
is convex and $f(0)\leq0$). This example can be easily extended by replacing
$\langle\mathbf{x},\mathbf{a}\rangle$ by any positive functional defined on an
ordered Banach space. See the Appendix for a quick review of the necessary
background on ordered Banach spaces.

The case of the function
\[
\Phi(\mathbf{x})=\Phi(x_{1},x_{2})=\sqrt{x_{1}x_{2}},\quad\text{for }%
x_{1},x_{2}\geq0,
\]
shows that Theorem \ref{thm_str_sub_1var} is not necessarily true in higher
dimensions. Indeed, $\Phi$ is concave, $\Phi(0,0)=0,$ but the condition
(\ref{eq_concave}) fails for $\mathbf{x}=(1/3,1/3),$ $\mathbf{y}=(1/3,2/3)$
and $\mathbf{z}=(0,0)$ since
\[
\Delta_{(1/3,1/3)}\Delta_{(1/3,2/3)}\Phi(0,0)=\sqrt{2/3}-\sqrt{1/9}-\sqrt
{2/9}+0\approx\allowbreak0.011\,758\,726\,\allowbreak8>0.
\]

An explanation for this failure is offered by the fact that the differential
of a differentiable concave function $\Phi$ of several variables is not
necessarily monotone decreasing. For example, in the case of the function
$\Phi(\mathbf{x})=\sqrt{x_{1}x_{2}},$ its differential $\Phi^{\prime}$ is
given by the formula%
\[
\Phi^{\prime}(\mathbf{x})=\left(  \frac{x_{2}}{2\sqrt{x_{1}x_{2}}},\frac
{x_{1}}{2\sqrt{x_{1}x_{2}}}\right)  ,\text{\quad}x_{1},x_{2}>0
\]
and $\Phi^{\prime}$ is neither nondecreasing nor nonincreasing as a function
from $\mathbb{R}_{++}^{2}$ into $L(\mathbb{R}^{2},\mathbb{R})$. Notice that
$S\leq T$ in $L(\mathbb{R}^{2},\mathbb{R})$ is understood as $S(\mathbf{z}%
)\leq T(\mathbf{z})$ for all $\mathbf{z\in}\mathbb{R}_{+}^{2}.$

As is well known (see for example \cite{NP2025}, Theorem 1.4.3), if $f$ is a
function continuous on the interval $I$ and differentiable on the interior of
$I,$ then $f$ is concave (respectively convex) iff its derivative is
nonincreasing (respectively nondecreasing). As a consequence, one obtains the
following variant of Theorem \ref{thm_str_sub_1var} under the presence of differentiability:

\begin{theorem}
\label{thm_str_sub_1variant_diff}Suppose that $\Phi$ is a function continuous
on $\mathbb{R}_{+}$ and differentiable on $\mathbb{R}_{++}$. Then $\Phi$ is
strongly subadditive if and only if $\Phi(0)\geq0$ and $\Phi^{\prime}$ is
nonincreasing $($respectively $\Phi$ is strongly superadditive if and only if
$\Phi(0)\leq0$ and $\Phi^{\prime}$ is nondecreasing$).$
\end{theorem}

This result will be extended in Section 2 to the framework of differentiable
functions of several variables. See Theorem \ref{thm_str_sub_Rn}.

As was already noticed, Theorem \ref{thm_str_sub_1var} does not extend to the
case of several variables. However, it admits an analog for twice
differentiable functions, based on the sign of the coefficients of the Hessian
matrix. See Theorem \ref{thm_str_sub_several var_2diff}.

Section 3 is devoted to a discussion concerning the connection between the
class of strongly subadditive functions and that of submodular functions
defined on the positive cone of a Banach lattice. It is shown (see Corollary
\ref{cor_str_submod}) that every strongly subadditive (respectively strongly
superadditive)\emph{ }function that belongs either to $C(\mathbf{R}_{+}%
^{N})\cap C^{2}(\mathbb{R}_{++}^{N})$ or to $C^{2}(\mathbb{R}^{N})$ is
submodular (respectively supermodular). The converse does not work in general,
an example being the log-sum-exp function from the theory of convex
optimization. See \cite{BV} and \cite{NP2025}. Surprisingly, if two of the
three variables appearing in the inequality of strong superadditivity are
comonotonic, then this inequality works for the log-sum-exp function. See
Remark \ref{rem_com_str_superad}.

The completely monotone functions verify the condition (\ref{eq_str_super}) of
2-monotone increasing, but not necessarily the property of superadditivity.
Using elementary geometric transformations, we will show in Section 4 how to
associate them strongly superadditive functions.

In Section 5, based on the weak majorization theorem of Tomi\'{c} and Weyl, we
show how to generate functional inequalities of Popoviciu type based on
strongly superadditive functions.

For the reader's convenience, the paper ends with an Appendix containing the
necessary background on ordered Banach spaces.

\section{The strong subadditivity/superadditivity in the context of several
variables}

In what follows $E$ and $F$ are two ordered Banach spaces.

Given a differentiable function $\Phi:E_{++}\rightarrow F,$ the sign of the
second order differences%
\[
\Delta_{\mathbf{x}}\Delta_{\mathbf{y}}\mathbf{\Phi}(\mathbf{z})=\left(
\Phi(\mathbf{x}+\mathbf{y}+\mathbf{z})-\Phi(\mathbf{x}+\mathbf{z})\right)
-\left(  \Phi(\mathbf{y}+\mathbf{z})-\Phi(\mathbf{z})\right)  ,
\]
can be settled via the monotonicity of the function $\Delta_{\mathbf{u}}\Phi$
$:\mathbf{v}\rightarrow\Phi(\mathbf{u}+\mathbf{v})-\Phi(\mathbf{u})$, which in
turn can be derived from the sign of the derivative of $\Delta_{\mathbf{u}%
}\Phi$ along every direction generated by a positive vector $\mathbf{w}\in
E_{++}$. See Lemma \ref{lemAmann_conv} in the Appendix. Since this derivative
is done by the formula%
\begin{align*}
\frac{\partial}{\partial\mathbf{w}}\Delta_{\mathbf{u}}\Phi(\mathbf{v})  &
=\lim_{\varepsilon\rightarrow0+}\frac{\Phi(\mathbf{u}+\mathbf{v+}%
\varepsilon\mathbf{w})-\Phi(\mathbf{u+}\varepsilon\mathbf{w})-\Phi
(\mathbf{u}+\mathbf{v})+\Phi(\mathbf{u})}{\varepsilon}\\
&  =\lim_{\varepsilon\rightarrow0+}\frac{\Phi(\mathbf{u}+\mathbf{v+}%
\varepsilon\mathbf{w})-\Phi(\mathbf{u}+\mathbf{v})}{\varepsilon}%
-\lim_{\varepsilon\rightarrow0+}\frac{\Phi(\mathbf{u+}\varepsilon
\mathbf{w})-\Phi(\mathbf{u})}{\varepsilon}\\
&  =d\Phi(\mathbf{u}+\mathbf{v})(\mathbf{w})-d\Phi(\mathbf{u})(\mathbf{w}),
\end{align*}
we are led to the following result:

\begin{theorem}
\label{thm_str_sub_Rn}Suppose that $\Phi:E_{++}\rightarrow F$ is a continuous
function, differentiable on $E_{++}.$ Then:

$(i)$ $\Phi$ is strongly subadditive if and only if $\Phi(\mathbf{0})\geq0$
and the differential of $\Phi$ is nonincreasing as a map from $E_{+}$ into
$L(E,F);$

$(ii)$ $\Phi$ is strongly superadditive if and only if $\Phi(\mathbf{0})\leq0$
and the differential of $\Phi$ is nondecreasing as a map from $E_{+}$ into
$L(E,F).$
\end{theorem}

\begin{corollary}
\label{ex4}For every $p\in(1,\infty)$, the function%
\[
\Phi:L^{p}\left(  \mathbb{R}\right)  \rightarrow\mathbb{R},\text{\quad}%
\Phi(f)=\left\Vert f\right\Vert _{p}^{p}=\int_{\mathbb{R}}\left\vert
f\right\vert ^{p}\mathrm{d}t,
\]
is strongly superadditive on the positive cone of the Banach lattice
$L^{p}\left(  \mathbb{R}\right)  $ $($of all $p$th power Lebesgue integrable
functions on $\mathbb{R}).$
\end{corollary}

\begin{proof}
Indeed, the function $\Phi$ is continuously differentiable and its
differential is defined by the formula%
\[
d\Phi(f)(h)=p\int_{\mathbb{R}}h\left\vert f\right\vert ^{p-1}%
\operatorname*{sgn}f\mathrm{d}t\text{\quad for all }f,h\in L^{p}\left(
\mathbb{R}\right)  .
\]
See \cite{NP2025}, Proposition $5.2.9$. Clearly,
\[
0\leq f\leq g\text{ in }L^{p}\left(  \mathbb{R}\right)  \text{ implies }%
d\Phi(f)\leq d\Phi(g)
\]
and due to the fact that $\Phi(0)=0$ the conclusion of the corollary follows
from Theorem \ref{thm_str_sub_Rn} $(ii).$
\end{proof}

\begin{corollary}
\emph{(}F. Zhang \cite{Zhang}, Section $7.2$, Exercise $36$\emph{)} The
function $\det$ is strongly superadditive on $\operatorname*{Sym}%
\nolimits^{+}(N,\mathbb{R}).$
\end{corollary}

See \cite{LS2014} for a generalization in the context of generalized matrix
functions and block matrices. The fact that $\det$ has positive differences of
any order on $\operatorname*{Sym}\nolimits^{+}(N,\mathbb{R})$ can be found in
\cite{NS2023}.

\begin{proof}
One applies Theorem \ref{thm_str_sub_Rn} $(ii)$. Indeed, the function
$X\rightarrow\det(X)$ is differentiable on $\operatorname*{Sym}\nolimits^{++}%
(N,\mathbb{R}))$ and its differential is given by the formula
\[
d\left(  \det(X)\right)  (C)=\frac{\partial\det(X)}{\partial C}=\det
(X)\operatorname*{trace}(CX^{-1}).
\]
for all $C\in\operatorname*{Sym}(N,\mathbb{R});$ this is a consequence of the
fact that
\[
\frac{\det\left(  X+tC\right)  -\det X}{t}=\det(X)\frac{\det(I+tCX^{-1})-1}%
{t}\rightarrow\det X\operatorname*{trace}CX^{-1}\text{ as }t\rightarrow0+.
\]

Next we will prove that $0\leq X\leq Y$ in $\operatorname*{Sym}\nolimits^{++}%
(N,\mathbb{R}))$ implies%
\[
d\left(  \det(X)\right)  (C)\leq d\left(  \det(Y)\right)  (C)
\]
for every $C\in\operatorname*{Sym}\nolimits^{+}(N,\mathbb{R}),$ equivalently,
that
\[
\det(X)\operatorname*{trace}(CX^{-1})\leq\det(Y)\operatorname*{trace}%
(CY^{-1})
\]
for all $C\in\operatorname*{Sym}\nolimits^{++}(N,\mathbb{R}).$ Use the trick
of replacing $C$ by $C+\varepsilon I$ and pass to the limit as $\varepsilon
\rightarrow0+.$

Put in a more convenient form, the last inequality reads as follows:%
\[
\det(C^{-1/2}XC^{-1/2})\operatorname*{trace}(C^{1/2}X^{-1}C^{1/2})\leq
\det(C^{-1/2}YC^{-1/2})\operatorname*{trace}(C^{1/2}Y^{-1}C^{1/2})
\]
for all $C\in\operatorname*{Sym}\nolimits^{++}(N,\mathbb{R}).$ This allows us
to reduce the proof of the strong superadditivity of the function $\det$ to
the following fact: if $U,V\in\operatorname*{Sym}\nolimits^{++}(N,\mathbb{R})$
and $U\leq V,$ then%
\begin{equation}
\det U\operatorname*{trace}U^{-1}\leq\det V\operatorname*{trace}V^{-1}.
\label{det super}%
\end{equation}
To prove this let $\lambda_{1}\geq\cdots\geq\lambda_{N}$ and $\mu_{1}%
\geq\cdots\geq\mu_{N}$ be the sequences of eigenvalues of $U$ and respectively
$V$ (counted with their multiplicities). According to Weyl's monotonicity
principle $($see the Appendix$)$ we have%
\begin{equation}
\lambda_{k}\leq\mu_{k}\quad\text{for }k=1,...,N. \label{ineq_eigenval}%
\end{equation}
In terms of eigenvalues, the inequality (\ref{det super}) reads as%
\[
\lambda_{1}\cdots\lambda_{N}\left(  \sum\nolimits_{k=1}^{N}\frac{1}%
{\lambda_{k}}\right)  \leq\mu_{1}\cdots\mu_{N}\left(  \sum\nolimits_{k=1}%
^{N}\frac{1}{\mu_{k}}\right)  ,
\]
which is a consequence of (\ref{ineq_eigenval}).

For an alternative proof of the strong superadditivity of the function $\det$,
see \cite{NP2025}, Lemma 6.5.1.
\end{proof}

In the same way one can prove the following result.

\begin{corollary}
If $A_{1},...,A_{n}$ are positive definite matrices of order $N,$ then the
function $\Phi:\mathbf{x\rightarrow}\log\det\left(  \sum\nolimits_{i=1}%
^{n}x_{i}A_{i}\right)  $ is strongly subadditive on $\mathbb{R}_{+}^{N}$.
\end{corollary}

We next exhibit some examples involving trace functions. If
$f:\mathbb{R\rightarrow R}$ is a continuously differentiable function, then
the formula%
\[
\Phi(A)=\operatorname*{trace}(f(A))
\]
defines a continuously differentiable function on the Hilbert space
$\operatorname*{Sym}(N,\mathbb{R})$ $($endowed with the Frobenius norm$)$. We
have%
\[
d\Phi(A)X=\operatorname*{trace}\left(  f^{\prime}(A)X\right)  \text{\quad for
all }A,X\in\operatorname*{Sym}(N,\mathbb{R}).
\]
As a consequence, if $f^{\prime}$ is also matrix monotone on $[0,\infty)$,
that is,
\[
A\leq B\text{ in }\operatorname*{Sym}\nolimits^{+}(N,\mathbb{R})\text{ implies
}f^{\prime}(A)\leq f^{\prime}(B)\text{ \ (in the Löwner order),}%
\]
then the function $\Phi$ has nonnegative differences of second order.
According to the Löwner-Heinz Theorem (see \cite{C2010}, Theorem 2.6), some
basic examples of matrix monotone functions on $[0,\infty)$ are the logarithm,
the function $-x^{r}$ for $r\in\lbrack-1,0]$ and $x^{r}$ for $r\in\lbrack
0,1]$. An immediate consequence is as follows.

\begin{corollary}
$($Bouhtou, Gaubert and Sagnol \emph{\cite{BGS2008}}, Lemma $5.3)$ The
function $\Phi(A)=\operatorname*{trace}A^{p}$ is strongly subadditive on
$\operatorname*{Sym}\nolimits^{+}(N,\mathbb{R})$ for $p\in\lbrack0,1]$ and
strongly superadditive for $p\in\lbrack1,2].$ Equivalently, it verifies the
inequality%
\[
\operatorname*{trace}(A+B+C)^{p}+\operatorname*{trace}C^{p}\leq
\operatorname*{trace}(A+C)^{p}+\operatorname*{trace}(B+C)^{p}\text{ for }%
p\in\lbrack0,1]
\]
and%
\[
\operatorname*{trace}(A+B+C)^{p}+\operatorname*{trace}C^{p}\geq
\operatorname*{trace}(A+C)^{p}+\operatorname*{trace}(B+C)^{p}\text{ for }%
p\in\lbrack1,2].
\]

\end{corollary}

In a similar way, the trace function associated to the function $f(t)=\int
_{0}^{t}(1+s^{p})^{1/p}ds$ is strongly superadditive on $\operatorname*{Sym}%
\nolimits^{+}(n,\mathbb{R})$ for $p\in(0,1].$ The fact that the derivative
$f^{\prime}(t)=(1+t^{p})^{1/p}$ is operator monotone on $[0,\infty)$ for
$p\in(0,1]$ is proved in a paper of Hansen \cite{Hansen}, page 810.

\begin{remark}
The \emph{von Neumann entropy},\emph{ }which is defined by the formula
\[
S(A)=-\operatorname*{trace}(A\log A)\text{\quad for }A\in\operatorname*{Sym}%
\nolimits^{+}(N,\mathbb{R}),
\]
represents an example of strongly subadditive trace function motivated by
Theorem \emph{\ref{thm_str_sub_Rn}} $(ii).$

A much deeper result in quantum theory is the strong subadditivity of the von
Neumann entropy for a tripartite quantum state. See Carlen \emph{\cite{C2010}%
}, for details and Audenaert, Hiai and Petz \emph{\cite{AHP}} for a generalization.
\end{remark}

\begin{remark}
$($The case of the function $\log\det)$ As is well known $($see
\emph{\cite{NP2025}}, Example $4.1.5\emph{)}$ the function $\log\det$ is
concave on $\operatorname*{Sym}\nolimits^{++}(n,\mathbb{R}).$ It is neither
subadditive nor super additive, but it verifies the condition
\[
\log\det\left(  A+B+C\right)  +\log\det C\geq\log\det\left(  A+C\right)
+\log\det\left(  B+C\right)  .
\]

See \emph{\cite{TCL2011}}, Lemma $2.3$. The connection of this inequality with
the entropy inequalities \emph{(}specialized to Gaussian distributed vectors
with prescribed covariances\emph{) }is discussed by Lami, Hirche, Adesso and
Winter \emph{\cite{LHAW2017}}.
\end{remark}

We pass now to the case of twice differentiable functions.

If $\Phi:\mathbb{R}_{+}^{N}\rightarrow\mathbb{R}$ is a function continuous on
$\mathbb{R}_{+}^{N}$ and twice differentiable on $\mathbb{R}_{++}^{N},$ then
for all points $\mathbf{x},\mathbf{z}\in\mathbb{R}_{+}^{N},$%
\[
\Delta_{x}\mathbf{\Phi}(\mathbf{z})=\mathbf{\Phi}(\mathbf{z}+\mathbf{x}%
)-\mathbf{\Phi}(\mathbf{z})=\int_{0}^{1}\langle\mathbf{x},\nabla\mathbf{\Phi
}(\mathbf{z}+t\mathbf{x})\rangle\mathrm{d}t.
\]
See \cite{DN2011}, Lemma 5.40, p. 264. As a consequence,%
\begin{align*}
\Delta_{\mathbf{y}}\Delta_{\mathbf{x}}\mathbf{\Phi}(\mathbf{z})  &
=\mathbf{\Phi}(\mathbf{z}+\mathbf{x}+\mathbf{y})-\mathbf{\Phi}(\mathbf{z}%
+\mathbf{x})-\mathbf{\Phi}(\mathbf{z}+\mathbf{y})+\mathbf{\Phi}(\mathbf{z}%
)=\Delta_{\mathbf{y}}\Delta_{\mathbf{x}}\mathbf{\Phi}(\mathbf{z})\\
&  =\Delta_{\mathbf{y}}\left(  \int_{0}^{1}\langle\mathbf{x},\nabla
\mathbf{\Phi}(\mathbf{z}+t\mathbf{x})\rangle\mathrm{d}t\right) \\
&  =\int_{0}^{1}\int_{0}^{1}\langle\nabla^{2}\mathbf{\Phi}(\mathbf{z}%
+s\mathbf{y}+t\mathbf{x})\mathbf{y},\mathbf{x}\rangle\mathrm{d}s\mathrm{d}t
\end{align*}
and this remark leads easily to the following analog of Theorem
\ref{thm_str_sub_1var} in the case of functions of several variables.

\begin{theorem}
\label{thm_str_sub_several var_2diff}Suppose that $\ \Phi\ $belongs either to
$C(\mathbf{R}_{+}^{N})\cap C^{2}(\mathbb{R}_{++}^{N})$ or to $C^{2}%
(\mathbb{R}^{N}).$ Then:

$(i)$ $\Phi$ is strongly subadditive if and only if $\Phi(\mathbf{0})\geq0$
and all partial derivatives of second order $\frac{\partial^{2}\Phi}{\partial
x_{i}\partial x_{j}}$ are nonpositive$;$

$(ii)$ $\Phi$ is strongly superadditive if and only if $\Phi(\mathbf{0})\leq0$
and all partial derivatives of second order $\frac{\partial^{2}\Phi}{\partial
x_{i}\partial x_{j}}$ are nonnegative.
\end{theorem}

An immediate consequence is the strong superadditivity on $\mathbf{R}_{+}^{N}$
of all polynomials functions $P(\mathbf{x})=P(x_{1},...,x_{N})$ with
nonnegative coefficients and a zero constant term. In particular, this is the
case of the Euclidean scalar product.

Also, according to Theorem \ref{thm_str_sub_several var_2diff}, the class of
strongly subadditive functions on $\mathbf{R}_{+}^{N}$ includes all functions
of the form $\Phi(\mathbf{x})=\sum\nolimits_{1\leq i<j\leq N}f_{i,j}%
(x_{i}-x_{j})$ for $f_{i,j}\in C^{2}(\mathbb{R})$ convex and all functions of
the form $\Psi(\mathbf{x})=f\left(  \sum\nolimits_{i=1}^{N}\lambda_{i}%
x_{i}\right)  -\sum\nolimits_{i=1}^{N}\lambda_{i}f(x_{i})$ for $f\in$
$C^{2}(\mathbb{R})$ concave and $(\lambda_{i})_{i=1}^{N}$ nonnegative weights.
Notice that in general $\Phi$ and $\Psi$ are not convex or concave.

\section{The connection with submodular functions}

In what follows $E$ denotes a Banach lattice and $\mathcal{C}$ is either its
positive cone $E_{+}$ or the entire space $E$. A function $\Phi:\mathcal{C}%
\rightarrow\mathbb{R}$ is called \emph{submodular} if
\[
\Phi\left(  \mathbf{x}\vee\mathbf{y}\right)  +\Phi\left(  \mathbf{x}%
\wedge\mathbf{y}\right)  \leq\Phi(\mathbf{x})+\Phi(\mathbf{y})\text{ for all
}\mathbf{x},\mathbf{y}\in\mathcal{C};
\]
$\Phi$ is called \emph{supermodular} if $-\Phi$ is submodular. Here
$\mathbf{x}\vee\mathbf{y}=\sup\left\{  \mathbf{x},\mathbf{y}\right\}  $ and
$\mathbf{x}\wedge\mathbf{y}=\inf\left\{  \mathbf{x},\mathbf{y}\right\}  $ (two
classical notations in the case of Banach lattices). A straightforward
characterization of these functions in the case of the Euclidean space endowed
with the coordinatewise order is as follows:

\begin{theorem}
\emph{(}Topkis' characterization theorem; see \emph{\cite{MR1990})} If
$\mathcal{C}$ is either the positive cone $\mathbb{R}_{+}^{N}$ or the entire
space $\mathbb{R}^{N}$ and $\Phi:\mathcal{C}\rightarrow\mathbb{R}$ is
continuous on $\mathcal{C}$ and continuously differentiable on
$\operatorname{int}\mathcal{C},$ then $\Phi$ is supermodular if and only if
\[
\frac{\partial^{2}\Phi}{\partial x_{i}\partial x_{j}}\geq0\text{\quad for all
}i\neq j.
\]

\end{theorem}

Combining this result with Theorem \ref{thm_str_sub_several var_2diff}, we
arrive at the following consequence.

\begin{corollary}
\label{cor_str_submod}Every strongly subadditive \emph{(}respectively strongly
superadditive\emph{) }function that belongs either to $C(\mathbf{R}_{+}%
^{N})\cap C^{2}(\mathbb{R}_{++}^{N})$ or to $C^{2}(\mathbb{R}^{N})$, is
submodular \emph{(}respectively supermodular\emph{)}.
\end{corollary}

The converse does not work in general. See the case of the log-sum-exp
function, defined on $\mathbb{R}^{N}$ by the formula
\[
\operatorname*{LSE}(\mathbf{x})=\log\left(  \frac{1}{N}\sum_{k=1}^{N}e^{x_{k}%
}\right)  ,\text{\quad}\mathbf{x}\in\mathbb{R}^{N}.
\]
Clearly,%
\[
\frac{\partial^{2}\operatorname*{LSE}}{\partial x_{i}\partial x_{j}%
}(\mathbf{x})=-\frac{e^{x_{i}}e^{x_{j}}}{\left(  \sum_{k=1}^{N}e^{x_{k}%
}\right)  ^{2}}<0\text{\quad for all }i\neq j,
\]
so by Topkis' characterization theorem, the function $\operatorname*{LSE}$ is
submodular. Moreover, $\operatorname*{LSE}(\mathbf{0})=0$ but this function is
not strongly subadditive. For example for $N=2$ and $x=(1.0),$ $y=(0,1)$ and
$z=(1,1)$ we have%
\begin{multline*}
\operatorname*{LSE}(\mathbf{x}+\mathbf{y}+\mathbf{z})+\operatorname*{LSE}%
(\mathbf{z})-\operatorname*{LSE}(\mathbf{x}+\mathbf{z})-\operatorname*{LSE}%
(\mathbf{y}+\mathbf{z})\\
=\log\left(  \frac{1}{2}\left(  e^{2}+e^{3}\right)  \right)  +\log\left(
\frac{1}{2}\left(  e+e\right)  \right) \\
-\log\left(  \frac{1}{2}\left(  e^{2}+e\right)  \right)  -\log\left(  \frac
{1}{2}\left(  e+e^{2}\right)  \right)  >\allowbreak0.379\,.
\end{multline*}

\begin{remark}
\label{rem_com_str_superad}Somehow surprisingly, the log-sum-exp function
verifies the following property of comonotonic strong superadditivity:
\[
\operatorname*{LSE}(\mathbf{x}+\mathbf{y}+\mathbf{z})+\operatorname*{LSE}%
(\mathbf{z})\geq\operatorname*{LSE}(\mathbf{x}+\mathbf{z})+\operatorname*{LSE}%
(\mathbf{y}+\mathbf{z})
\]
for all $\mathbf{x},\mathbf{y},\mathbf{z}\in\mathbb{R}_{+}^{N}$ such that
$\mathbf{x}$ and $\mathbf{y}$ are comonotonic. Recall that two vectors
$\mathbf{u},\mathbf{v}\in\mathbb{R}^{N}$ are \emph{comonotonic} if%
\[
\left(  u_{i}-u_{j}\right)  \left(  v_{i}-v_{j}\right)  \geq0\text{\quad for
all indices }i,j\in\{1,...,N\}.
\]

Comonotonicity appeared in non-linear expected utility theory and is genuine
in the context of Choquet's integrability. See \emph{\cite{Denn}} and
\emph{\cite{GN2021}}. Two simple cases of comonotonicity are
\[
u_{1}\leq\cdots\leq u_{n},\text{ }v_{1}\leq\cdots\leq v_{n}\text{ and }%
u_{1}\geq\cdots\geq u_{n},\text{ }v_{1}\geq\cdots\geq v_{n}.
\]

The property of comonotonic strong superadditivity of the function
$\operatorname*{LSE}$ is a direct consequence of Chebyshev's algebraic
inequality$:$
\[
\langle\mathbf{u},\mathbf{p}\rangle\langle\mathbf{v},\mathbf{p}\rangle
\leq\langle\mathbf{uv},\mathbf{p}\rangle
\]
whenever $\mathbf{u},\mathbf{v}\in\mathbb{R}^{N}$ are comonotonic and
$\mathbf{p}=(p_{1},...,p_{N})$ is a probability vector \emph{(}that is,
$p_{i}\geq0$ for all $i$ and $\sum\nolimits_{i=1}^{N}p_{i}=1).~$See
\emph{\cite{NP2025}}, Section $2.2$, Exercise $8$.
\end{remark}

An important source of submodular functions is provided by the theory of
Choquet integral. See \cite{GN2021} for a brief presentation of this integral
(including a discussion on its submodularity).

\section{Strong superadditivity from complete monotonicity}

An important class of strongly superadditive functions motivated by Theorem
\ref{thm_str_sub_several var_2diff} is that of completely monotone functions.
See Schilling, Song and Vondra\v{c}ek \cite{SSV} and Scott and Sokal
\cite{SS2014} for details.

Recall that an infinitely many differentiable function $\Phi:$ $(0,\infty
)^{N}\rightarrow\lbrack0,\infty)$ is called \emph{completely monotone} if%
\begin{equation}
\left(  -1\right)  ^{k}\frac{\partial^{k}\Phi}{\partial x_{i_{1}}\partial
x_{i_{2}}\cdots\partial x_{i_{k}}}(\mathbf{x})\geq0\label{eq_cm_RN}%
\end{equation}
for all $\mathbf{x}\in\left(  0,\infty\right)  ^{N}$, all integers $k\geq1$
and all choices of indices $i_{1},\cdots\imath_{k}$; equivalently, if%
\begin{equation}
(-1)^{k}\Delta_{\mathbf{x}_{1}}\Delta_{\mathbf{x}_{2}}\cdots\Delta
_{\mathbf{x}_{k}}f(\mathbf{x})\geq0\label{eq_cm_differences}%
\end{equation}
for all vectors $\mathbf{x},\mathbf{x}_{1},...,$ $\mathbf{x}_{k}$
$\in(0,\infty)^{N}$ and all integers $k\geq1.$ A function $\Phi$ with this
property may be unbounded at the origin. If it is continuous at the origin we
say that it is \emph{completely monotone }on $[0,\infty)^{N}$.

It is straightforward to prove that for all $\alpha_{1},...,\alpha_{N}>0$ the
function
\[
e^{-\alpha_{1}x_{1}-\cdots-\alpha_{N}x_{N}}%
\]
is completely monotone on $[0,\infty)^{N},$ while the function $x_{1}%
^{-\alpha_{1}}x_{2}^{-\alpha_{2}}\cdots x_{N}^{-\alpha_{N}}$ is completely
monotone on $(0,\infty)^{N}$.

Due to the characterization (\ref{eq_cm_differences}), every completely
monotone function is 2-monotone increasing. If a completely monotone function
attains the value 0, then it vanishes everywhere, so except for this trivial
case it cannot be strongly superadditive. However, composing it with a
translation and subtracting the value at the origin one obtains a strongly
superadditive function.

\begin{example}
According to \emph{\cite{SS2014}}, Example $2.5$. p. $343$, if $a>0,$ then the
function
\[
f_{\beta}(x)=(1+ae^{-x})^{\beta}%
\]
is completely monotone on $[0,\infty)$ if and only if $\beta\in\{0,1,2,...\}.$
However, according to Theorem \emph{\ref{thm_str_sub_1var}} the function
$f_{\beta}(x)-(1+a)^{\beta}$ is strongly superadditive for all $\beta
\in\{0\}\cup\lbrack1,\infty).$
\end{example}

\begin{example}
According to \emph{\cite{SS2014}}, Corollary $1.6$, p. $329$, the function
\[
\Phi(\mathbf{x})=\left(  x_{1}x_{2}+x_{1}x_{3}+x_{1}x_{4}+x_{2}x_{3}%
+x_{2}x_{4}+x_{3}x_{4}\right)  ^{-\beta}%
\]
is completely monotone on $(0,\infty)^{4}$ if and only if $\beta=0$ or
$\beta\geq1.$ Denoting by $\mathbf{1}$ the vector in $\mathbb{R}^{4}$ whose
components equal unity, we infer that for the same values of the exponents
$\beta$ the function
\[
\Psi(\mathbf{x})=\Phi(\mathbf{x+1})-\Phi(\mathbf{1})
\]
is strongly superadditive on $[0,\infty)^{4}.$
\end{example}

For more examples we need the concept of complete monotonicity for functions
defined on cones more general than $\mathbb{R}_{++}^{N}.$

Let $V$ denote a finite-dimensional real vector space and $\mathcal{C}$ an
open convex cone in $V$ with closure \ $\overline{\mathcal{C}}$. The dual cone
of $\mathcal{C}$ is defined by the formula $\mathcal{C}^{\ast}=\{\mathbf{y}\in
V^{\ast}:\langle\mathbf{y},\mathbf{x}\rangle\geq0$ for all $\mathbf{x}%
\in\mathcal{C}\}$. Thus $\mathcal{C}$ consists of all the linear functionals
that are nonnegative on $\overline{\mathcal{C}}$.

The dual cone of $[0,\infty)^{N}$ in $\mathbb{R}^{N}$ is $[0,\infty)^{N}$ and
the dual cone of $\operatorname*{Sym}\nolimits^{+}(N,\mathbb{R})$ in
$\operatorname*{Sym}(N,\mathbb{R})$ is again $\operatorname*{Sym}%
\nolimits^{+}(N,\mathbb{R})$; see Example 2.24, Section 2.6 of Boyd and
Vandenberghe's Convex Optimization \cite{BV}.

\begin{definition}
\label{def_complete_mon}A function $f:\mathcal{C}\rightarrow$ $\mathbb{R}_{+}$
is called completely monotone%
\index{function!completely monotone}
if $f$ is $\mathcal{C}^{\infty}$ on $\mathcal{C}$ and, for all integers
$k\geq1$ and all vectors $\mathbf{v}_{1},...,$ $\mathbf{v}_{k}$ $\in
\mathcal{C}$, we have%
\[
\left(  -1\right)  ^{k}\partial_{\mathbf{v}_{1}}\cdots\partial_{\mathbf{v}%
_{k}}f(\mathbf{x})\geq0\text{\quad for all }\mathbf{x}\in\mathcal{C}.
\]
Here $\partial_{\mathbf{v}}$ denotes the directional derivative along the
vector $\mathbf{v}$.

A function $f:\overline{\mathcal{C}}\rightarrow$ $\mathbb{R}_{+}$ is called
completely monotone if it is the continuous extension of a completely monotone
function on $\mathcal{C}$.
\end{definition}

An important result in the theory of completely monotone functions on cones is
the following representation theorem.

\begin{theorem}
\emph{(Bernstein-Hausdorff-Widder-Choquet theorem)}.\label{thm_BHWCh} Let
$\mathcal{C}$ be an open convex cone in a finite-dimensional real vector
space. A continuous function $\Phi:\overline{\mathcal{C}}\mathcal{\rightarrow
}\mathbb{R}_{+}$ is completely monotone if and only if it is the Laplace
transform of a unique Borel probability measure $\mu$ supported on the dual
cone $\mathcal{C}^{\ast}$, that is,%
\begin{equation}
\Phi(\mathbf{x})=\int_{\mathcal{C}^{\ast}}e^{-\langle\mathbf{x},\mathbf{y}%
\rangle}\mathrm{d}\mu(\mathbf{y})\text{\quad for all }\mathbf{x}\in
\overline{\mathcal{C}}. \label{eq_BHWC}%
\end{equation}

\end{theorem}

For the proof, see Choquet \cite{Cho1969}. Notice that the "only if" part is
nontrivial, the other part resulting from the dependence of the integral on
parameters. Thus, the existence of an integral representation like
(\ref{eq_BHWC}) serves as a certificate for the complete monotonicity of
$\Phi$. See \cite{SS2014} and \cite{KMS2019}.

An immediate consequence is as follows:

\begin{corollary}
\label{lem_str_superad}Let $\mathcal{C}$ be an open convex cone in a
finite-dimensional real vector space. Then every completely monotone function
$\Phi:\overline{\mathcal{C}}\rightarrow\mathbb{R}_{+}$ is $2$-monotone increasing.
\end{corollary}

\begin{proof}
According to the Bernstein-Hausdorff-Widder-Choquet theorem, the function
$\Phi$ admits the representation%
\[
\Phi(\mathbf{x})=\int_{\mathcal{C}^{\ast}}e^{-\langle\mathbf{x},\mathbf{u}%
\rangle}\mathrm{d}\mu(\mathbf{u}),\text{\quad}\mathbf{x}\in\overline
{\mathcal{C}},
\]
for a suitable Borel measure $\mu$ supported on the dual cone $\mathcal{C}%
^{\ast}.$ Therefore, for all $\mathbf{x},\mathbf{y},\mathbf{z}\in\mathcal{C},$%
\begin{multline*}
\Phi(\mathbf{x}+\mathbf{y}+\mathbf{z})+\Phi(\mathbf{z})-\Phi(\mathbf{x}%
+\mathbf{z})-\Phi(\mathbf{y}+\mathbf{z})\\
=\int_{\mathcal{C}^{\ast}}e^{-\langle\mathbf{u},\mathbf{x}+\mathbf{y}%
+\mathbf{z}\rangle}\mathrm{d}\mu(\mathbf{u})+\int_{\mathcal{C}^{\ast}%
}e^{-\langle\mathbf{u},\mathbf{z}\rangle}\mathrm{d}\mu(\mathbf{u})\\
-\int_{\mathcal{C}^{\ast}}e^{-\langle\mathbf{u},\mathbf{x}+\mathbf{z}\rangle
}\mathrm{d}\mu(\mathbf{u})-\int_{\mathcal{C}^{\ast}}e^{-\langle\mathbf{u}%
,\mathbf{y}+\mathbf{z}\rangle}\mathrm{d}\mu(\mathbf{u})\\
=\int_{\mathcal{C}^{\ast}}\left[  e^{-\langle\mathbf{u},\mathbf{x}%
+\mathbf{y}+\mathbf{z}\rangle}+e^{-\langle\mathbf{u},\mathbf{z}\rangle
}-e^{-\langle\mathbf{u},\mathbf{x}+\mathbf{z}\rangle}-e^{-\langle
\mathbf{u},\mathbf{y}+\mathbf{z}\rangle}\right]  \mathrm{d}\mu(\mathbf{u}%
)\geq0\\
=\int_{\mathcal{C}^{\ast}}e^{-\langle\mathbf{u},\mathbf{z}\rangle}\left(
e^{-\langle\mathbf{u},\mathbf{x}\rangle}-1\right)  \left(  e^{-\langle
\mathbf{u},\mathbf{y}\rangle}-1\right)  \mathrm{d}\mu(\mathbf{u})\geq0,
\end{multline*}
and the proof is done.
\end{proof}

According to \cite{SS2014}, Theorem 1.3, p. $327$, the function $\left(  \det
A\right)  ^{-\beta}$ (defined for $A\in\operatorname*{Sym}\nolimits^{++}%
(N,\mathbb{R}))$ is completely monotone if and only if $\beta\in\{0,\frac
{1}{2},1,\frac{3}{2},...\}\cup\lbrack\frac{N-1}{2},\infty).$ As a consequence,
for the same values of the exponents, the function $\Phi_{\beta}(A)=\left(
\det(I+A)\right)  ^{-\beta}-1$ is strongly superadditive.

The complete monotonicity of the function $\left(  \det A\right)  ^{-\beta}$
is established by Scott and Sokal \cite{SS2014}, pp. 355-356, providing
suitable integral representations of the type (\ref{eq_BHWC}). The same works
for the function $\Phi_{\beta}.$ For example, in the case $\beta=1/2,$%
\[
\frac{1}{\left(  \det(I+A)\right)  ^{1/2}}=\int_{\mathbb{R}^{N}}e^{-\langle
Ax,x\rangle}\prod\nolimits_{i=1}^{N}\frac{e^{-x_{i}^{2}}dx_{i}}{\sqrt{\pi}},
\]
where $A$ is a real symmetric positive-definite $N\times N$ matrix.

\section{An application to functional inequalities}

The aim of this section is to prove an inequality verified by the strongly
superadditive functions. The main ingredient is the classical theorem of
Tomi\'{c} and Weyl about the weak majorization. See \cite{NP2025}, Theorem 6.1.4.

\begin{theorem}
\emph{(}Theorem on weak majorization\emph{)}\label{thmTW} Let $f$ be a
nondecreasing convex function defined on a nonempty interval $I.$ If
$(a_{k})_{k=1}^{n}$ and $(b_{k})_{k=1}^{n}$ are two families of numbers in $I$
such that%
\[
\sum_{k=1}^{m}a_{k}\leq\sum_{k=1}^{m}b_{k}\quad\text{for }m=1,\dots,n,
\]
If $a_{1}\geq\cdots\geq a_{n},$ then
\[
\sum_{k=1}^{n}f(a_{k})\leq\sum_{k=1}^{n}f(b_{k}),
\]
while when $b_{1}\leq\cdots\leq b_{n},$ then the conclusion works in the
reverse direction.
\end{theorem}

With the help of Theorem \ref{thmTW}, we can obtain the following new
functional inequality of Popoviciu type.

\begin{theorem}
\label{thm_Pop_gen}Suppose that $\mathcal{C}$ is a convex cone and
$\Phi:\mathcal{C\rightarrow}\mathbb{R}_{+}$ is a strongly superadditive
function such that%
\[
\mathbf{x}\leq\mathbf{y}\text{ in }\mathcal{C}\text{ implies }\Phi
(\mathbf{x})\leq\Phi(\mathbf{y}).
\]
Then for every $\mathbf{x},\mathbf{y},\mathbf{z}\in\mathcal{C}$ and every
nondecreasing and convex function $f$ defined on an interval including the
range of $\Phi,$ we have the inequality%
\begin{equation}
f\left(  \Phi(\mathbf{x}+\mathbf{y}+\mathbf{z})\right)  +f\left(
\Phi(\mathbf{z})\right)  \geq f\left(  \Phi(\mathbf{x}+\mathbf{z})\right)
+f\left(  \Phi(\mathbf{y}+\mathbf{z})\right)  , \label{eq_Pop_type}%
\end{equation}
which, by symmetrization, leads to the following Popoviciu type inequality
\emph{\cite{Pop1965}},
\begin{multline*}
\frac{f\left(  \Phi(\mathbf{x})\right)  +f\left(  \Phi(\mathbf{y})\right)
+f\left(  \Phi(\mathbf{z})\right)  }{3}+f\left(  \Phi(\mathbf{x}%
+\mathbf{y}+\mathbf{z})\right) \\
\geq\frac{2}{3}\left(  f\left(  \Phi(\mathbf{x}+\mathbf{y})\right)  +f\left(
\Phi(\mathbf{y}+\mathbf{z})\right)  +f\left(  \Phi(\mathbf{x}+\mathbf{z}%
)\right)  \right)  .
\end{multline*}

If $f$ is nonincreasing and concave, then the last inequality works in the
other direction.
\end{theorem}

\begin{proof}
Without loss of generality we may assume $a_{1}=\Phi(\mathbf{y}+\mathbf{z}%
)\geq a_{2}=\Phi(\mathbf{x}+\mathbf{z}).$ Then
\[
a_{1}=\Phi(\mathbf{y}+\mathbf{z})\leq b_{1}=\Phi(\mathbf{x+y+z})
\]
and%
\[
a_{1}+a_{2}=\Phi(\mathbf{y}+\mathbf{z})+\Phi(\mathbf{x}+\mathbf{z})\leq
b_{1}+b_{2}=\Phi(\mathbf{x+y+z})+\Phi(\mathbf{z})
\]
and thus the conclusion follows from Theorem \ref{thmTW}.\ \ \ \ \ \ \ \ \ \ \ \ \ \ \ \ \ \ \ \ \ \ \ \ \ \ \ \ \ \ \ \ \ \ \ \ \ \ \ \ \ \ \ \ \ \ \ \ \ \ \ \ \ \ \ \ \ \ \ \ \ \ \ \ \ \ \ \ \ \ \ \ \ \ \ \ \ \ \ \ \ \ \ \
\end{proof}

\begin{corollary}
\ \ \ \ \ \ For all $A,B,C\in\operatorname*{Sym}\nolimits^{+}(N,\mathbb{R})$
and $p\geq1,\,$\ the following inequality holds:\ \
\begin{multline*}
\left(  \det A\right)  ^{p}\ +(\det C)^{p}\ +\det\left(  A+B+C\right)  ^{p}\\
\geq\frac{2}{3}\left[  \left(  \det(A+B)\right)  ^{p}+\left(  \det
(B+C)\right)  ^{p}+\left(  \det(A+C)\right)  ^{p}\right]  .
\end{multline*}

\end{corollary}

\begin{remark}
The proof of Theorem \emph{\ref{thm_Pop_gen} }also shows that its
conclusion\emph{ }remains true if $\Phi$ is comonotonic strongly superadditive
and $\mathbf{x,y}$ and $\mathbf{z}$ are pairwise
comonotonic.$\ \ \ \ \ \ \ \ \ \ \ $\ \ \ \ \ \ \ \ \ \ \ \
\end{remark}

See \cite{N2021}, Theorem $10$, for a genuine generalization of Popoviciu's
inequality to the framework of several variables.

\section{Appendix. Preliminaries on ordered Banach spaces}

We next present a quick review of the necessary background on ordered Banach
spaces. For more, see the paper \cite{GN2024a} and the textbook \cite{AT2007}.

Recall that an ordered vector space is any pair $(E,\allowbreak C)$ consisting
of a real vector space $E$ and a convex cone $C\subset E$ such that
$C\cap\left(  -C\right)  =\left\{  0\right\}  $. It is customary to denote an
ordered vector space by its underlying vector space.

In any ordered vector space we may define a partial ordering (i.e., a
reflexive, antisymmetric, and transitive relation) $\leq$ on $E$ by defining
$x\leq y$ if and only%
\[
y-x\in C;
\]
$x\leq y$ will be also denoted $y\geq x$ and $x<y$ (equivalently $y>x$) will
mean that $y\geq x$ and $x\neq y.$

Notice that $x\in C$ if and only if $x\geq0$; for this reason $C$ is also
denoted $E_{+}$ and called the \emph{positive cone} of $E.$

The concept of ordered Banach space is usually defined as any ordered vector
space $E$ endowed with a complete norm (sometimes adding the condition that
the positive cone is closed in the norm topology). See \cite{AT2007}, p. 85.
However, for the purposes of the present paper we will need a more restrictive
concept of ordered Banach space. More precisely, we will consider \emph{only}
those ordered Banach spaces $E$ which also verify the following two properties:

\begin{enumerate}
\item[(OBS1)] the cone $E_{+}$ is \emph{generating}, that is, $E=E_{+}-E_{+};$

\item[(OBS2)] the norm is\emph{ monotone on the positive cone, }that is,%
\[
0\leq x\leq y\text{ implies }\left\Vert x\right\Vert \leq\left\Vert
y\right\Vert .
\]

\end{enumerate}

An ordered vector space $E$ such that every pair of elements $x,y$ admits a
supremum $\sup\{x,y\}$ and an infimum is called a \emph{vector lattice}. In
this case for each $x\in E$ we can define $x^{+}=\sup\left\{  x,0\right\}  $
(the positive part of $x$), $x^{-}=\sup\left\{  -x,0\right\}  $ (the negative
part of $x$) and $\left\vert x\right\vert =\sup\left\{  -x,x\right\}  $ (the
modulus of $x$). $\ $We have $x=x^{+}-x^{-}$ and $\left\vert x\right\vert
=x^{+}+x^{-}.$ A Banach lattice is any real Banach space $E$ which is at the
same time a vector lattice and verifies the compatibility condition%
\[
\left\vert x\right\vert \leq\left\vert y\right\vert \text{ implies }\left\Vert
x\right\Vert \leq\left\Vert y\right\Vert .
\]

Most (but not all) ordered Banach spaces are actually Banach lattices. So are
the Euclidean space $\mathbb{R}^{N}$ and the discrete spaces $c_{0},$ $c$ and
$\ell^{p}$ for $1\leq p\leq\infty$ (endowed with the coordinate-wise
ordering). The same is true for the function spaces
\begin{align*}
C(K)  &  =\left\{  f:K\rightarrow\mathbb{R}:f\text{ continuous on the compact
Hausdorff space }K\right\} \\
C_{b}(X)  &  =\left\{  f:X\rightarrow\mathbb{R}:\text{ }f\text{ continuous and
bounded on the metric space }X\right\}  ,\\
\mathcal{U}C_{b}(X)  &  =\left\{  f:X\rightarrow\mathbb{R}:\text{ }f\text{
uniformly continuous and bounded}\right.  \text{ on the}\\
&  \left.  \text{metric space }X\right\}  ,\\
C_{0}(X)  &  =\left\{  f:X\rightarrow\mathbb{R}:\text{ }f\text{ continuous and
null to infinity on the locally }\right. \\
&  \left.  \text{compact Hausdorff space }X\right\}  ,
\end{align*}
each one endowed with the sup-norm $\left\Vert f\right\Vert _{\infty}=\sup
_{x}\left\vert f(x)\right\vert $ and the pointwise ordering. Other Banach
lattices of an utmost interest are the Lebesgue spaces $L^{p}(\mu)$
($p\in\lbrack1,\infty]),$ endowed with norm%
\[
\left\Vert f\right\Vert _{p}=\left\{
\begin{array}
[c]{cl}%
\left(  \int_{X}\left\vert f(x)\right\vert ^{p}\mathrm{d}\mu(x)\right)  ^{1/p}
& \text{if }p\in\lbrack1,\infty)\\
\operatorname*{esssup}\limits_{x\in X}\left\vert f(x)\right\vert  & \text{if
}p=\infty
\end{array}
\right.
\]
and the pointwise ordering modulo null sets.

Some important examples of ordered Banach spaces (that are not necessarily
Banach lattices) are provided by the Schatten classes. The \emph{real Schatten
class} $\operatorname{Re}S_{p}(\mathbb{R}^{N})$ is just the space
$\operatorname*{Sym}(N,\mathbb{R)},$ of all $N\times N$-dimensional symmetric
matrices with real coefficients, when endowed with the \emph{Löwner ordering},%
\[
A\leq B\text{ if and only if }\langle A\mathbf{x},\mathbf{x}\rangle\leq\langle
B\mathbf{x},\mathbf{x}\rangle\text{ for all }\mathbf{x}\in\mathbb{R}^{N},
\]
and the \emph{Schatten trace norm} $\left\Vert \cdot\right\Vert _{S_{p}},$
which is defined by the formula
\[
\left\Vert A\right\Vert _{S_{p}}=\left\{
\begin{array}
[c]{cl}%
\left(  \sum_{k=1}^{N}\left\vert \lambda_{k}(A)\right\vert ^{p}\right)  ^{1/p}
& \text{if }p\in\lbrack1,\infty)\\
\sup_{1\leq k\leq N}\left\vert \lambda_{k}(A)\right\vert  & \text{if }%
p=\infty.
\end{array}
\right.
\]
Here $\lambda_{1}(A),...,\lambda_{N}(A)$ is the list of the eigenvalues of
$A$, repeated according to their multiplicities. It is worth noticing that the
Schatten norm of index $\infty$ equals the \emph{operator norm},
\[
\left\Vert A\right\Vert =\sup_{\left\Vert \mathbf{x}\right\Vert \leq
1}\left\vert \langle A\mathbf{x},\mathbf{x}\rangle\right\vert ,
\]
while the Schatten norm of index $2$ is nothing but the \emph{Frobenius
norm},
\[
\left\Vert A\right\Vert _{F}=\left(  \sum_{i=1}^{N}\sum_{j=1}^{N}a_{ij}%
^{2}\right)  ^{1/2},
\]
provided that $A=\left(  a_{ij}\right)  _{i,j=1}^{N}.$ Notice that the
Frobenius norm is associated to the inner product%
\[
\langle A,B\rangle=\operatorname*{trace}(AB).
\]

It is usual to denote $\operatorname{Re}S_{2}(\mathbb{R}^{N})$ simply by
$\operatorname*{Sym}(N,\mathbb{R}).$

The positive cone of each of the spaces $\operatorname{Re}S_{p}(\mathbb{R}%
^{N})$ is
\[
\operatorname*{Sym}\nolimits^{+}(N,\mathbb{R})=\left\{  A\in
\operatorname*{Sym}(N,\mathbb{R}):\langle A\mathbf{x},\mathbf{x}\rangle
\geq0\text{ for all }\mathbf{x}\in\mathbb{R}^{N}\right\}  ,
\]
while its interior equals%
\[
\operatorname*{Sym}\nolimits^{++}(N,\mathbb{R})=\left\{  A\in
\operatorname*{Sym}(N,\mathbb{R}):\langle A\mathbf{x},\mathbf{x}%
\rangle>0\text{ for all }\mathbf{x}\in\mathbb{R}^{N},\mathbf{x}\neq0\text{
}\right\}  .
\]

As is well known, $A\geq0$ (respectively $A>0$) if and only if its spectrum
$\sigma(A)$ is included in $[0,\infty)$ (respectively in $(0,\infty)$).

Sometimes it is convenient to list the eigenvalues of a symmetric matrix
$C\in\operatorname*{Sym}(N,\mathbb{R})$ in decreasing order and repeated
according to their multiplicities as follows:%
\[
\lambda_{1}^{\downarrow}(C)\geq\cdots\geq\lambda_{N}^{\downarrow}(C).
\]
So is the case of \emph{Weyl's monotonicity principle }$($see \cite{NP2025},
Corollary 6.4.3, p. 253$)$, which asserts that%
\[
A\leq B\text{ in }\operatorname*{Sym}(N,\mathbb{R})\text{ implies }\lambda
_{i}^{\downarrow}(A)\leq\lambda_{i}^{\downarrow}(B)\text{ for }i=1,...,N.
\]
An immediate consequence of it is the inequality%
\[
0\leq A\leq B\text{ in }\operatorname{Re}S_{p}(\mathbb{R}^{N})\text{ implies
}\left\Vert A\right\Vert _{p}\leq\left\Vert B\right\Vert _{p},
\]
which shows that the real Schatten classes are indeed ordered Banach spaces.

The theory of Schatten classes associated to infinite dimensional separable
Hilbert spaces (real or complex) can be found in the book of Simon
\cite{Simon2005}.

As was noticed by Amann \cite{Amann1974}, Proposition 3.2, p. 184, the Gâteaux
differentiability offers a convenient way to recognize the monotone increasing
functions acting on ordered Banach spaces: the positivity of their differential:

\begin{lemma}
\label{lemAmann_conv}Suppose that $E$ and $F$ are two ordered Banach spaces,
$C$ is a closed convex subset of $E$ with nonempty interior
$\operatorname{int}C$ and $\Phi:C\rightarrow F$ is a function continuous on
$C$ and Gâteaux differentiable on $\operatorname{int}C.$ Then $\Phi$ is
increasing on $C$ if and only if $\Phi^{\prime}(\mathbf{a})\geq0$ for all
$\mathbf{a}\in\operatorname{int}C.$
\end{lemma}

\begin{proof}
The "only if" part follows immediately from the definition of the Gâteaux
derivative. For the other implication, notice that the gradient inequality
mentioned by Lemma $2$ shows that $\Phi$ is isotone on $\operatorname{int}C$
if $\Phi^{\prime}(\mathbf{a})\geq0$ for all $\mathbf{a}\in\operatorname{int}%
C$. As concerns the isotonicity on $C,$ that follows by an approximation
argument. Suppose that $\mathbf{x},\mathbf{y}\in C$ and $\mathbf{x}%
\leq\mathbf{y}$. For $\mathbf{x}_{0}\in\operatorname*{int}C$ arbitrarily fixed
and $t\in\lbrack0,1),$ both elements $\mathbf{u}_{t}=\mathbf{x}_{0}%
+t(\mathbf{x}-\mathbf{x}_{0})$ and $\mathbf{v}_{t}=\mathbf{x}_{0}%
+t(\mathbf{y}-\mathbf{x}_{0})$ belong to $\operatorname*{int}C$ and
$\mathbf{u}_{t}\leq\mathbf{v}_{t}.$ Moreover, $\mathbf{u}_{t}\rightarrow
\mathbf{x}$ and $\mathbf{v}_{t}\rightarrow\mathbf{y}$ as $t\rightarrow1.$
Passing to the limit in the inequality $\Phi(\mathbf{u}_{t})\leq
\Phi(\mathbf{v}_{t})$ we conclude that $\Phi(\mathbf{x})\leq\Phi(\mathbf{y}).$
\end{proof}

\begin{remark}
\label{rem3}If the ordered Banach space $E$ has finite dimension, then the
statement of Lemma \emph{\ref{lemAmann_conv}} remains valid by replacing the
interior of $C$ by the relative interior of $C$. See \emph{\cite{NP2025}},
Exercise $6$, p. $81$.
\end{remark}

\bigskip

\noindent\textbf{Acknowledgements}. The author would like to thank \c{S}tefan
Cobza\c{s} and Dan-\c{S}tefan Marinescu for many useful comments on the
subject of this paper.

\bigskip

\noindent\textbf{Declarations}

\smallskip

\noindent\textbf{Funding} The author declares that no funds, grants, or other
support were received during the preparation of this manuscript.

\noindent\textbf{Data Availability Statement} Data sharing is not applicable
to this article as no datasets were generated or analyzed during the current study.

\noindent\textbf{Conflict of interest} The author has no conflict of interests
to report.

\end{document}